\RequirePackage{fix-cm}
\documentclass[twocolumn]{article}
\usepackage{graphicx}
\usepackage{amsmath}
\usepackage{amsthm}

\newtheorem{theorem}{Theorem}[section]
\newtheorem{corollary}{Corollary}[theorem]

\newcommand{\keywords}[1]{\textbf{\textit{Keywords:}} #1}
\begin{document}

\title{Graph partition method based on finite projective planes
}


\author{Oleg Kruglov\footnote{was Huawei employee while working on this project} \and
	Anna Mastikhina\footnote{Corresponding author,\\ mastikhina.anna.antonovna@huawei.com} \and
	Oleg Senkevich \and
	Dmitry Sirotkin \and
	Stanislav Moiseev*\\
	Huawei Technologies Co., Ltd.
}

\maketitle

\begin{abstract}
We present a novel graph partition algorithm with a theoretical bound for the replication factor of $\sqrt{n}$, which improves known constrained approaches (grid: 2$\sqrt{n}$-1, torus: 1.5$\sqrt{n}$+1) and provides better performance.
\end{abstract}


\label{seq:introduction}

Recently, the use of large graphs has become increasingly popular in various real-world applications, including web and social network services, as well as machine learning and data mining algorithms. In order to perform distributed graph computing on clusters with multiple machines or servers, the entire graph must be partitioned across these machines. The quality of the graph partitioning significantly impacts the performance of any algorithm operating on the graph.

In practice, graphs are often distributed across multiple machines to optimize storage and improve computation. However, the way in which the graph is partitioned can significantly impact load balance and the number of interactions between machines, which in turn affects runtime performance. Unfortunately, achieving a perfect partition without any data duplication is typically impossible.

The edge-cut partitioning method has been commonly used in the past. It was effective in computational grids and used to solve differential equations. However, finding the optimal solution to minimize the edge-cut in a graph is an NP-hard problem, even for the simplest bisection subcase. Therefore, heuristic approaches have been used since the 1970s (Klin, 1973). In contrast, vertex partitioning has been found to be much more effective for real-world graphs with skewed distributions, such as those with power-law distributions. Therefore, this paper will focus on the vertex-cut approach.

There are numerous partitioning algorithms \\
described in the literature and implemented in frameworks that, when used in conjunction with graph algorithms, demonstrate significant performance improvements (\cite{adwise}, \cite{dbh}, \cite{hdrf}, \cite{hep}). Also, there exist various graph processing systems available that not only provide standard graph analytics algorithms but also incorporate built-in partitioning algorithms. Examples of such systems include PowerGraph and PowerLyra \cite{powerlyra} families, GraphX, GraphBuilder, and the Pregel family.

In this paper, we will use GraphX for partition quality evaluation. GraphX is a distributed graph engine that is developed using Spark (\cite{graphx}). It enhances Resilient Distributed Graphs (RDGs), which are used to associate records with vertices and edges. GraphX uses a flexible vertex-cut partitioning approach to encode graphs as horizontally partitioned collections by default. The number of partitions is equivalent to the number of blocks (64 Mb) in the input file. GraphX is widely used for graph analytics both in research and the industry. More specifically, we will use algorithms based on a GraphX Pregel API for the method evaluation \cite{graphx_cost}.

Partitioning methods are typically designed to improve certain metrics, with the motivation behind them being taken from real-world scenarios. The most common and important metrics are:

\begin{enumerate}
	
	\item Balance
	\begin{equation}\label{eq:balance}
		B = \frac{\max_{i} \mid E_i \mid}{\frac{\mid E\mid}{m}}
	\end{equation}
	\item Replication Factor (RF)
	\begin{equation}\label{eq:rf}
		RF=\sum_{i=1}^{m} \frac{\mid V(E_i )\mid}{\mid V\mid}
	\end{equation}
\end{enumerate}

Balance value refers to how evenly the edges are
 distributed among machines in a graph partitioning scheme. An ideal balance is achieved when each machine has an equal number of edges, resulting in a balance value of 1.0. The replication factor, on the other hand, measures the average number of vertex replicas in the partitioning. Minimizing the replication factor reduces the amount of communication between machines. However, the extent of improvement achieved by minimizing the replication factor depends on various factors, including the specific algorithm being used, the software implementation, the structure of the cluster, the characteristics of the graph, and other considerations. Consequently, there is no single algorithm that can universally optimize these metrics for all scenarios.

In addition, it should be noted that partition methods aimed at achieving better partition quality (for example, \cite{ne}) often involve more complex calculations. While these methods may yield improved partitioning results, they also tend to be more time-consuming. This becomes particularly important in scenarios where the partitioning process needs to be performed before each application run. In such cases, the overall computation pipeline can be significantly affected by the time required for the partitioning step. Consequently, using complicated graph partition algorithms may not be feasible in scenarios where the partitioning process is an integral part of the entire computation pipeline.

Hash-based methods, also known as hash-based partitioning, offer a time-efficient alternative to graph partitioning. These methods rely on pre-aggregated information about the edges and their endpoints to determine the partition assignment. While they may utilize less detailed information about the graph structure compared to other approaches, they can still achieve a reduction in vertex replications. The advantage of hash-based methods lies in their ability to perform the partitioning process quickly, making them suitable for scenarios where time efficiency is crucial.

This article presents a novel approach to graph partitioning called finite-geometry-based hash-based partitioning. Our method introduces advancements in both the theoretical bound of the replication factor and the overall performance acceleration. 

\subsection*{Article structure}

The article is structured as follows: Chapter \ref{seq:constrained} will discuss the so-called constrained approach for graph partitioning and the current state-of-the-art. Chapter \ref{seq:FPP} will introduce a novel partitioner called FPP partitioner. Chapter \ref{seq:empirical_results} will provide evaluation of this method and chapter \ref{seq:conclusion} is a conclusion

\section{Constrained approach for the graph partitioners} \label{seq:constrained}

\subsection{Finite geometries and finite projective planes}

Finite geometry is a branch of mathematics that deals with geometrical objects and spaces but with a finite number of elements. One of the key aspects of finite geometry is the study of finite planes, which are mathematical structures that possess a finite number of points and lines. In this article, we focus on a specific type of finite plane: the projective plane.

Projective geometry is a type of geometry that is concerned with properties that are invariant under projective transformations. Projective transformations,\\
 also known as homographies, are mappings that preserve incidence and collinearity. In projective geometry, parallel lines intersect at a point called the "point at infinity," which results in the addition of a new point to the geometry. A projective plane is a mathematical structure that consists of points and lines, with the following properties:

\begin{itemize}
	\item Any two distinct points lie on exactly one line.
	\item Any two distinct lines intersect at exactly one point.
	\item There exists a set of four points, no three of which are collinear.
\end{itemize}

In a finite projective plane of order n, there are $n^2 + n + 1$ points and $n^2 + n + 1$ lines. Each line contains $n + 1$ points, and each point is incident with $n + 1$ lines. Moreover, any two distinct lines intersect at exactly one point, and any two distinct points are connected by exactly one line. The number $n$ is referred to as the \textit{order} of the corresponding projective plane. For all known finite projective planes its order $n$ is a prime power.

Consider a 3-dimensional vector space $V$ over the finite field $F_q$, where $q$ is a prime power $p^k$. In this vector space, there are $q^3$ elements.

The projective plane $P^2$ can be defined as the set of one-dimensional lines in $V$. To represent points on the projective plane, we use homogeneous coordinates in the form of a tuple $(u_0, u_1, u_2)$, where at least one element should be a non-zero one. This tuple represents a set of points on the corresponding line, given by all scalar multiples of $(u_0, u_1, u_2)$, where the scalar $\lambda$ is non-zero.

If we assume $u_0 \in F_q$ is non-zero, then every vector \\ $(u_0, u_1, u_2)$ is collinear with a vector of the form\\
 $(1, v_1, v_2)$. There are $q^2$ such vectors. Similarly, every vector $(0, u_0, u_1)$ is collinear with a vector of the form $(0, 1, v_1)$. There are $q$ such vectors. Lastly, every vector $(0, 0, u_0)$ is collinear with the vector $(0, 0, 1)$. There is only one such vector. From these observations, we can conclude that any finite projective plane over $F_q$ contains $q^2 + q + 1$ points.

Each line in the projective plane can be defined by a point $u$, and it contains all points $v$ such that their scalar product in $V$ (given by $v_1u_1 + v_2u_2 + v_3u_3$) equals zero.

Using the same reasoning, it can be proven that there are also $q^2 + q + 1$ projective lines in $P^2$. Furthermore, each line in the projective plane contains exactly $q + 1$ points. Lastly, any two lines in $P^2$ intersect at exactly one point, as the intersection of two distinct two-dimensional planes in $V$ corresponds to a single one-dimensional line (point in $P^2$).

For further information and a more detailed analysis of projective planes, you can refer to the $\cite{jkarl}$.

Finite geometries, such as projective planes, find applications in various tasks such as packing \cite{packing} and partitioning problems. These constructions can be utilized to establish connections between edge and vertex assignments. For instance, a mapping can be defined from vertices to lines in the projective plane, and edges can be assigned based on the intersection index of the lines corresponding to the incident vertices. This mapping ensures that each intersection is nonempty, allowing for a valid assignment.

Additionally, replication can be evaluated using this approach. In this partitioning scheme, a vertex cannot be assigned to more partitions than the number of points on a line. This limitation helps control the replication factor.

This approach will be discussed in detail in section \ref{seq:FPP}.

\subsection{Constrained approach}

Among the various partitioning methods available in the literature, our focus in this study will be on a specific subset known as \textit{constrained methods} \cite{torus}. These methods have gained attention due to their ability to provide an upper bound on the replication factor. They are referred to as "constrained" because they guarantee the upper bound of the replication factor.

In general, the constrained partitioning methods follow the same schema.

To illustrate the partitioning process, we start with a set of partitions denoted as $S$, where the cardinality of $S$ is $n$. Our goal is to find subsets $S_i$ that satisfy the condition $S_i \cap S_j \neq \emptyset$.

Each partition can be represented by a mapping function $\psi: V \rightarrow \{S_i\}$, where $V$ represents the set of vertices in the graph. For any given edge $(u, v)$, we consider the subsets $\psi(u)$ and $\psi(v)$. It is assumed that these subsets have non-empty intersections, allowing us to choose a partition for this edge from the intersection.

The size of each subset $\mid S_i\mid$ determines the number of replications, as the edges incident to a vertex $v$ with $\psi(v) = S_i$ cannot be assigned to more partitions than the number of elements in $S_i$.

\begin{figure}
	\centering
	\includegraphics[width=0.45\textwidth]{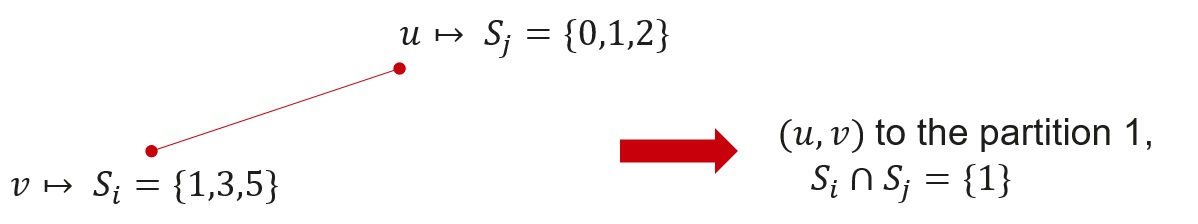}
	\caption{\label{fig:constr}Constrained approach}
\end{figure}

\subsubsection{EdgePartition2D}

One of the state-of-the-art methods that we consider as a baseline for comparison is the EdgePartition2D partitioner. This partitioning technique is implemented as "EdgePartition2D" in GraphX and referred to as the "grid" partitioner in PowerGraph. The EdgePartition2D method guarantees a replication factor of less than $2\sqrt{n}-1$, where $n$ represents the number of partitions. This serves as a benchmark against which we can evaluate the performance of our proposed method.

We can partition the adjacency matrix of the graph by splitting it into equally sized square blocks. The number of blocks is denoted as $n$, and they are arranged in $\sqrt{n}$ rows and $\sqrt{n}$ columns. Each block represents a partition, and the edges that are located within the same block belong to the corresponding partition. Specifically, for an edge $(i, j)$, belongs to the partition determined by the index of the block in which the matrix element $(i, j)$ is situated.

In cases where the desired number of partitions is not a perfect square, some partitions may be merged together to achieve the desired number. This ensures that all edges are assigned to a partition and that the partitions are as balanced as possible. This method is hash-based since it distributes the edge based on the hashes of its ending vertices.

\begin{figure}[h]
	\centering
	\includegraphics[width=0.45\textwidth]{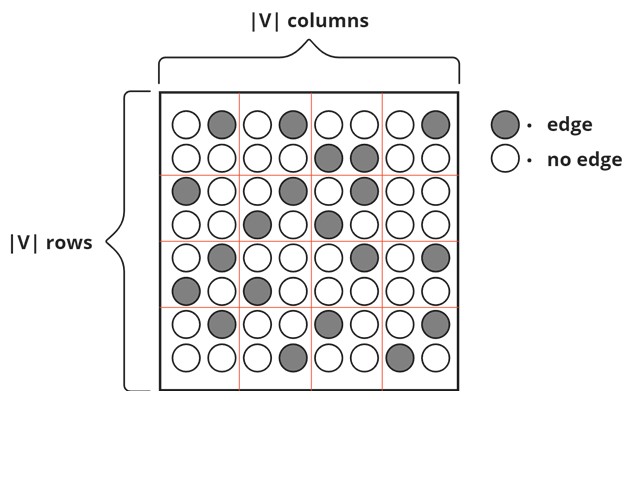}
	\caption{\label{fig:e2d_1} Edge2D partition: edge distribution}
\end{figure}

Let us say that set of blocks $S_i$ consists of $i$-th line and $i$-th column for the blocks in the adjacency matrix. We denote a set of indexes of blocks from $S_i$ as $s_i$. Overall, each of the sets $S_i$ contains $2\sqrt{n} - 1$ blocks and correspondingly each of the sets $s_i$ contains $2\sqrt{n} - 1$ numbers. Edge incident to vertices from $i$-th row of blocks and $j$-th column of blocks are assigned to the partition index from the intersection of $s_i$ and $s_j$.

Therefore we can obtain a well-known theorem

\begin{theorem}
	The replication factor of EdgePartition2D is bounded by  $2\sqrt{n}-1$.
\end{theorem}

\begin{proof}
	Consider one vertex from $i$-th block. All the \\ edges which are adjacent to this vertex can only be contained in the $i$-th row of blocks and $i$-th column of blocks (e.g. inset $S_i$). So any single vertex can’t be repeated more than $\mid S_i\mid = 2\sqrt{n}-1$ times.
\end{proof}

\begin{figure}[h]
	\centering
	\includegraphics[width=0.45\textwidth]{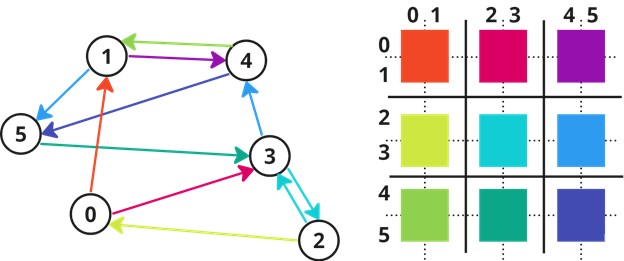}
	\caption{\label{fig:e2d_2} Edge2D graph partition}
\end{figure}

\subsubsection{Torus}

Another example of a constrained approach is a Torus-based edge partition (\cite{torus}).

Let's consider a torus-like grid, which resembles the structure of the EdgePartition2D method and consists of $n$ blocks. Each block of the grid corresponds to a partition in the graph. We define $n$ subsets $S_i$, where each subset contains approximately $r \approx 1.5\sqrt{n}+1$ blocks.

Each subset $S_i$ is composed of a complete column and half of a row in the grid. In the case where the row extends beyond the left side of the grid, it continues on the right side, taking into account the torus-like nature of the grid. An example of such a subset, denoted as $S_{13}$, is shown in Figure \ref{fig:torus_part}.

\begin{figure}[h]
	\centering
	\includegraphics[width=0.45\textwidth]{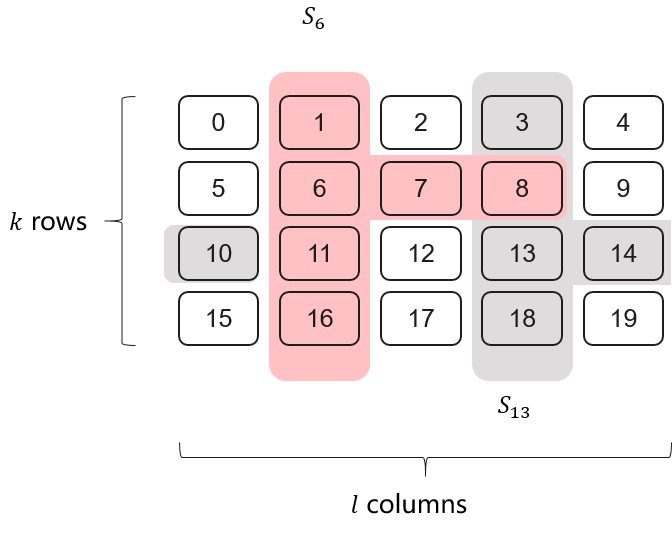}
	\caption{\label{fig:torus_part}Torus-based partition}
\end{figure}

We define a mapping function $\psi: V \to \{S_i\}$ that assigns each vertex $v$ to one of the subsets $S_i$. For any given edge $(u, v)$, we consider the two subsets $\psi(u)$ and $\psi(v)$ to which the vertices $u$ and $v$ are assigned. Since these subsets have a non-empty intersection, we can choose the partition for this edge from the common blocks in the intersection.

It is worth noting that the maximal replication factor for this partitioning method is approximately $r \approx 1.5\sqrt{n}+1$, assuming that $k=l=\sqrt{n}$. This replication factor is lower than the approximate replication factor achieved by the EdgePartition2D method. So, the replication factor has an upper bound $1.5\sqrt{n}+1$.

\subsection{An improvement of a current theoretical bound for the replication factor in the constrained approach}

In this section, we will establish a theoretical upper bound for the replication factor for the constrained approach. 

\begin{theorem}\label{th:repl_bound}
	Consider the set $S$ of the size $n$.\\
	 Let $S_1, ..., S_n$ be subsets of $S$ such that, $S_i \cap S_j \ne \emptyset$ for each $i,j \in \{1,...,n\}, i \neq j$ and for every $s\in S$ there exist no more than $r$ indices $i_1, ..., i_r$ such that $s \in S_{i_k}, k=1,...r$.
	
	Then $r \ge \sqrt{n}$
\end{theorem}
\begin{proof}
	Consider $\sum_{i<j} \mid S_i\cap S_j\mid$. 
	
	Since $\mid S_i\cap S_j\mid \ge 1$ and number of such pairs is $C_n^2$, $\sum_{i<j} \mid S_i\cap S_j\mid\ge C_n^2 = \frac{n(n-1)}{2}$.
	On the other hand, $C_n^2 \leq \sum_{i<j} \mid S_i\cap S_j\mid = \sum_{s\in S}\mid\{(i,j): i<j, s\in S_i \cap S_j\}\mid \leq \sum_{s\in S} C_r^2 = n\cdot \frac{r(r-1)}{2}$ since each element $s$ belongs to no more than $r$ subsets.
	
	So, 
	
	$$
	r(r-1)\ge n-1
	$$
	$$
	r^2 \ge n-1 + r \ge n
	$$
	
	$$
	r \ge \sqrt{n}
	$$
	
\end{proof}

Within the boundaries of a constrained approach, we assign edges to partitions based on a mapping $V \to S$ for any set of partitions $S$ with $\mid S\mid \leq \mid V\mid$. Theorem \ref{th:repl_bound} provides an upper bound for the replication factor of such a partition, which is lesser or equal than $r$, with $r \ge \sqrt{n}$.

We aim to demonstrate that, for the complete graph $K_{m} = (V, E)$, any partitioning method is unable to significantly enhance the bound of $\sqrt{n}$. Suppose we have an edge partition $E=E_{1} \cup ... \cup E_n$, and we define $\alpha$ as the load balance, which can be calculated 
as $\alpha$ = $\min_{i}\frac{\mid E_{i}\mid\cdot n}{\mid E\mid}$.

\begin{theorem}
	
	Let us consider an edge partition of a\\ complete graph $K_m = (V, E)$ into $n$ parts $E_1 \cup \dots \cup E_n$ with a load balance of $\alpha$, defined
	as $\alpha = \min_{i} \frac{\mid E_{i}\mid \cdot n}{\mid E\mid}$.\\
	The replication factor of the partition can be bounded as follows: $RF \geq \sqrt{\alpha} \cdot \sqrt{n} \cdot \sqrt{\frac{m-1}{m}}$.
\end{theorem}

\begin{proof}
	
	By definition the replication factor is\\
	 $$
	 RF = \frac{\sum_i \mid V(E_{i})\mid}{\mid V\mid}.
	 $$
	
	Let's evaluate the number of vertices $V(E_{i})$ in the partition $E_{i}$. In any graph with $k$ vertices, the number of edges is bounded by $\mid E_{i}\mid \leq \frac{k(k-1)}{2}$. Therefore, we can derive the inequality $k(k-1) \geq 2\mid E_{i}\mid$. Therefore for each i $\mid V(E_{i})\mid \ge \sqrt{2\mid E_{i}\mid}$.
	
	If we consider a complete graph with $m$ vertices, the total number of edges in the graph is given by $\frac{m(m-1)}{2}$. Therefore, the average number of edges in each partition can be calculated as $\frac{m(m-1)}{2n}$.
	
	Let us define the balance parameter $b_{i} = \frac{\mid E_{i}\mid}{\mid E\mid/n}$. We can rewrite it as $\mid E_{i}\mid = \frac{\mid E\mid}{n} \cdot b_{i}$.
	
	Then $\mid V(E_{i})\mid \ge \sqrt{\frac{2\mid E\mid}{n}\cdot b_{i}}$. Let us consider RF:
	
	\begin{equation*}
	\begin{split}
	\frac{\sum_i \mid V(E_i)\mid}{\mid V\mid}\ge \frac{1}{m} \sqrt{\frac{2\mid E\mid}{p}}(\sqrt{b_1} + ... + \sqrt{b_n})
	\ge \\
	\ge \frac{1}{m} \sqrt{\frac{2\mid E\mid}{n}} \sqrt{\alpha} n\\
	\end{split}
	\end{equation*}
	
	For the complete graph $\mid E\mid = \frac{m(m-1)}{2}$ and therefore:
	
	\begin{equation*}
	\begin{split}
	\frac{\sum_i \mid V(E_i)\mid}{\mid V\mid}\ge \frac{1}{m} \sqrt{\frac{2m(m-1)}{2n}} \sqrt{\alpha}n
	=\\
	= \sqrt{\frac{m-1}{m}} \sqrt{\alpha} \sqrt{n}
    \end{split}
    \end{equation*}

	In practice, this implies that when aiming for a balance close to 1 and dealing with a large number of vertices, it is not feasible to achieve a replication factor ($RF$) significantly less than $\sqrt{n}$. In order to reduce the replication factor, one have to sacrifice balance by assigning a majority of the edges to a single partition.
	
\end{proof}

\begin{corollary}
	Let's denote by $RF(m)$ the lower bound of the replication factor for any partition into $n$ parts with a load balance $\alpha$ for a complete graph with $m$ vertices. As the number of vertices $m$ approaches infinity, the lower bound $RF(m)$ converges to $\sqrt{\alpha} \cdot \sqrt{n}$.
\end{corollary}

\section{FPP Partition algorithm}\label{seq:FPP}

In our method, the partition is based on the structure of a projective plane. The set $S$ represents the projective plane, where each point corresponds to an index of a partition, and the subsets $S_i$ are the lines on the plane. The number of subsets equals the total number of partitions, denoted as $n$, which can be calculated as $n=p^{2k}+p^k+1$, where $p$ is a prime number. It is important to note that the intersection of any two subsets $S_i$ contains exactly one point.

Each vertex in the graph is associated with a specific subset through a mapping $\psi:V\to{S_i}$. Furthermore, all subsets consist of $r=p^k+1$ points, ensuring that no vertex is replicated more than $p^k+1$ times. It is worth mentioning that $p^k+1$ is approximately equal to $\sqrt{p^{2k}+p^k+1}=\sqrt{n}$.

This lower bound on replication factor, \\
approximately equal to $\sqrt{n}$, indicates that it cannot be further improved. Thus, the proposed method achieves a replication factor close to its lower bound.

\subsection{Partitioning algorithm description}

\begin{figure}
	\centering
	\includegraphics[width=0.45\textwidth]{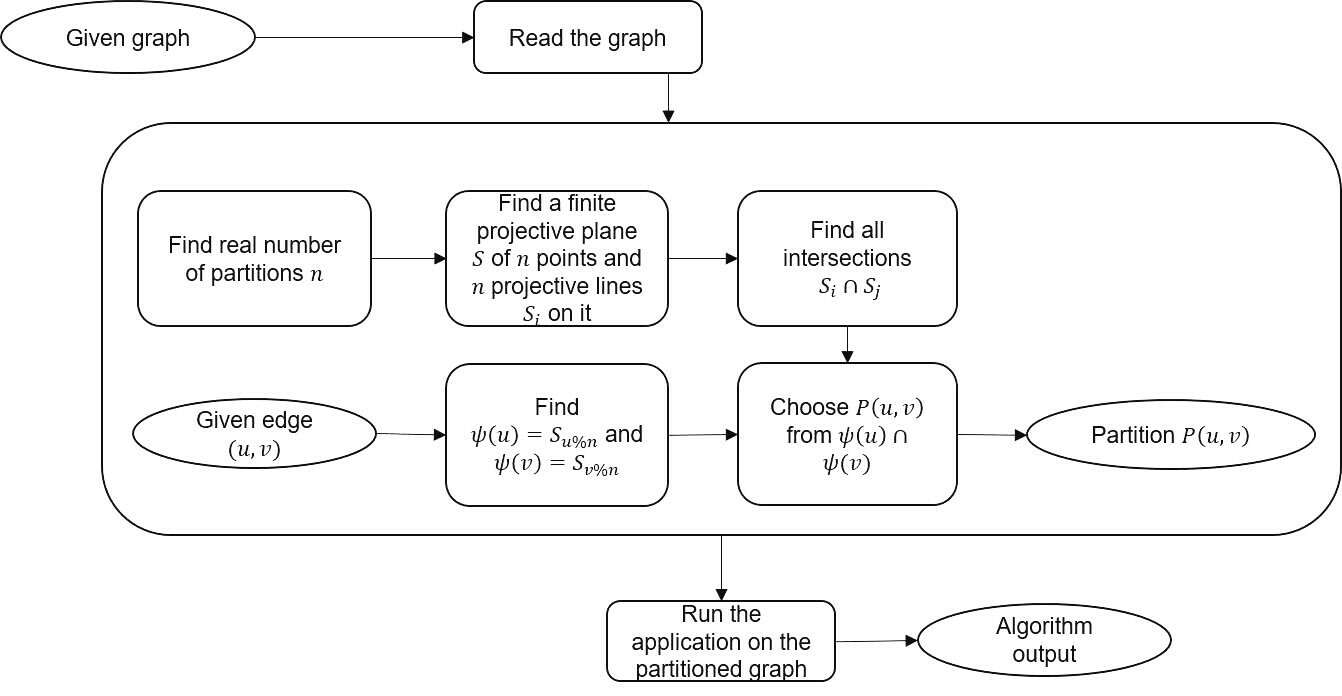}
	\caption{\label{fig:flow}Flow of FPP partition algorithm.}
\end{figure}

Our method operates as follows:

\begin{enumerate}
	\item
	Receive the set of edges and the number of partitions $n = p^{2k} + p^k + 1$.
	\item
	Form the set $S$ of points in the projective plane, where each point is represented as an array with 3 elements from the finite field $F$ of order $p$.
	\item
	Represent each projective line as a 2-dimensional subspace in a 3-dimensional vector space over the finite field $F$.
	\item
	Define each line by a point $u$, which consists of all points $v$ such that $u \cdot v = 0$, where the scalar product is defined over $F$.
	\item
	Find and store the intersections of subsets.
	\item
	Construct the function $\psi: V \to {1, \ldots, n}$ as the remainder of division by $n$: $\psi(v) = v \mod n$.
	\item
	Denote the number of partitions to which an edge $e = {u, v}$ is assigned as $P(u, v)$.
	Assign each edge to a partition according to the following rule: $P(u, v) = i$, where $i \in S_{\psi(u)} \cap S_{\psi(v)}$.
	\item
	Since $\mid S_i \cap S_j\mid = 1$ for all $i \neq j$, such assignment is always possible.
\end{enumerate}

For the case $\psi(v) = \psi(u)$, we need to choose a partition index from $S_{\psi(v)}$ which has $p^k + 1$ elements. This method will be discussed in detail in subsection \ref{dfpp}.

A flowchart of the algorithm is illustrated in Figure \ref{fig:flow}.

\subsection{Theoretical bound}

FPP partition algorithm guarantees the upper bound of replication factor $\sqrt{n}$ which can't be improved in the general case.

\begin{theorem}
	Consider a graph $G = (V, E)$ with $n = p^{2k} + p^k + 1$ for some prime number $p$ and non-negative integer $k$. Let $RF$ denote the replication factor of a partition into $n$ parts using the FPP (Finite Projective Plane) partition approach. Then $RF \leq RF_{max}(n)$, where $RF_{max}(n) \approx \sqrt{n}$.
	
\end{theorem}

\begin{proof}
	
	Each vertex in the partitioning is associated with a specific subset through the mapping $\psi: V \to {S_i}$. Consequently, the partition number for an edge ${u,v}$ is chosen from the intersection of the subsets $S_{\psi(u)}$ and $S_{\psi(v)}$. This implies that each edge incident to vertex $v$ can only be assigned to partitions from $S_{\psi(v)}$. Furthermore, since all subsets consist of $p^k+1$ points, it follows that no vertex can be replicated more than $p^k+1$ times.
	
	It is worth noting that		\\
	$p^k = \sqrt{p^{2k}+p^k+1 - p^k-1} = \sqrt{n-p^k - 1} \leq \sqrt{n}$.
	
	On the other hand, \\
	$p^k + 1 = \sqrt{p^{2k} + 2p^k + 1} = \sqrt{n + p^k} \geq \sqrt{n}$.
	
	Hence, we have \\
	$\sqrt{n} \leq p^k + 1 \leq \sqrt{n} + 1$.
	
\end{proof}

In the proof, we assumed that the number of partitions is given by the formula $n = p^{2k} + p^k + 1$ for a prime number $p$. However, if the desired number of partitions does not match this formula, we can still adapt the method by considering the largest number $n' < n$ that satisfies the equation.

In this case, the excess partitions, i.e., the difference $n - n'$, can be handled in two ways. First, they can be left empty, resulting in $n'$ as the actual number of partitions. Alternatively, these excess partitions can be used with any suitable distribution to ensure a balanced partitioning. 

\subsection{Determined FPP partition algorithm} {\label{dfpp}}

If $\psi(u)$ and $\psi(v)$ represent the same projective line, in the standard variant of FPP Partition, we would choose a partition randomly from this line. However, we can eliminate this randomness by constructing a one-to-one mapping $\phi: {S_i} \to S$. This mapping ensures that we select one point from each set of points on the line in a deterministic manner.

To establish this mapping, we create a bipartite graph $\overline{G}=(V, U, E), V\cap U = \emptyset, E\subseteq \{\{u,v\}, u\in U, v\in V\}$, where $V$ represents the projective points and $U$ represents the projective lines. An edge from $E$ exists between two vertices, $v \in V$ and $u \in U$, if and only if the corresponding point $v$ lies on the corresponding line $u$ in the projective plane.

In a bipartite graph, a \textit{matching} refers to a subset of edges in the graph where no two edges share a common vertex. A \textit{perfect matching} is matching that covers all vertices in the graph, meaning that every vertex is incident to exactly one edge in the matching.

A one-to-one mapping $\phi : {S_i} \to S$ that associates each line with a point can be established by finding a perfect matching in the bipartite graph $\overline{G}$.

We can find a perfect matching in the bipartite graph $\overline{G}$ using Kuhn's algorithm, as described in \cite{kuhn}. It is important to note that both parts of the graph have equal size $n = p^{2k} + p^k + 1$, and all vertices have the same degree $p^k + 1$. By applying Kuhn's algorithm, we can guarantee the existence of a perfect matching in the graph. This property follows by Kőnig's theorem \cite{konig1916}, which states that in a bipartite graph with equal-sized parts and vertices of the same degree, a perfect matching always exists.

This perfect matching in $\overline{G}$ allows us to construct the desired mapping $\phi$. If $\psi(u) = \psi(v)$, then the edge $(u,v)$ is assigned to the partition $\phi(\psi(u)) = \phi(\psi(v))$. Since $\phi(\psi(u)) \in \psi(u)$, this assignment does not violate the theoretical upper bound of the FPP Partition. 

The image of mapping $\phi$ covers all partitions, which ensures that the partition remains balanced. Additionally, the replication factor for the determined FPP is slightly reduced compared to the regular FPP. This reduction occurs because all edges $(u,v)$ where $\psi(u) = \psi(v)$ are assigned to the same partition.

\subsection{Example}

\begin{figure}
	\centering
	\includegraphics[width=0.35\textwidth]{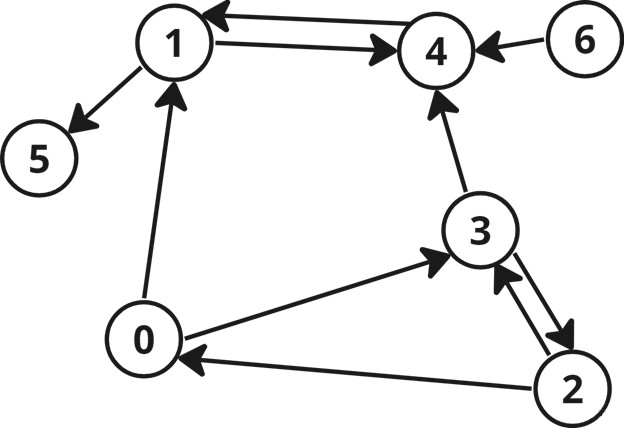}
	\caption{\label{fig:graph} Example: graph.}
\end{figure}

Consider the graph shown in Figure \ref{fig:graph}.

Let's construct a finite projective plane over the field $\mathbf{Z}_2$. In this construction, every homogeneous linear equation with coefficients $(u_1:u_2:u_3)$ defines a line in the projective plane.

\begin{figure}
	\centering
	\includegraphics[width=0.2\textwidth]{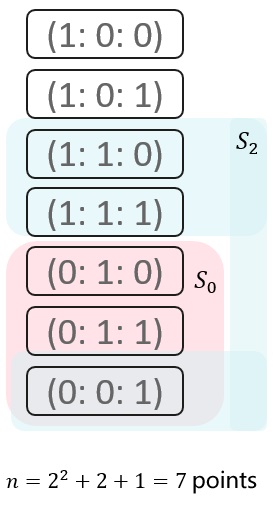}
	\caption{\label{fig:lines} Example of FPP partition.}
\end{figure}

\begin{equation*} \label{eq1}
	\begin{split}
		u_1=0, S_0 & = \{(0:1:0), (0:1:1), (0:0:1)\} \\
		u_1+u_3=0, S_1 & = \{(1:0:1),(1:1:1),(0:1:0)\} \\
		u_1+u_2=0, S_2 & = \{(1:1:0), (1:1:1),(0:0:1)\} \\
		u_1+u_2+u_3=0, S_3 & = \{(1:0:1), (1:1:0), (0:1:1)\} \\
		u_2=0, S_4 & = \{(1:0:0),(1:0:1), (0:0:1)\} \\
		u_2+u_3=0, S_5 & = \{(1:0:0),(1:1:1), (0:1:1)\} \\
		u_3=0, S_6 & = \{(1:0:0),(1:1:0),(0:1:0)\}
	\end{split}
\end{equation*}

Picture \ref{fig:lines} shows the intersections of subsets $S_i$ for the graph from Figure \ref{fig:graph}.	

Let's define $P(u, v)$ as the partition number assigned to the edge $(u, v)$.

\begin{equation*} \label{eq2}
	\begin{split}
		S_0\cap S_1 & = \{(0:1:0)\} \Longrightarrow P(0, 1)=1 \\
		S_0 \cap S_3 & = \{(0:1:1)\} \Longrightarrow P(0, 3)=2 \\
		S_1 \cap S_5 & = \{(1:1:1)\} \Longrightarrow P(1, 5)=6 \\
		S_1\cap S_4 & = \{(1:0:1)\} \Longrightarrow P(1, 4)=P(4, 1)=4 \\
		S_0\cap S_2 & = \{(0:0:1)\} \Longrightarrow P(2, 0)=0 \\
		S_2\cap S_3 & = \{(1:1:0)\} \Longrightarrow P(2, 3)=P(3, 2)=5 \\
		S_3\cap S_4 & = \{(1:0:1)\} \Longrightarrow P(3, 4)=4 \\
		S_4\cap S_6 & = \{(1:0:0)\} \Longrightarrow P(6, 4)=3
	\end{split}
\end{equation*}

The replication factor can be calculated as the average number of partitions to which each vertex belongs. In this case, the replication factor is $\frac{3+3+2+3+2+1+1}{7} \approx 2.14$, which is less than $\sqrt{7}$.

\section{Empirical results}\label{seq:empirical_results}
\subsection{Metrics calculations}

We analyzed two metrics, balance and replication factor (RF), and observed their improvements in Table \ref{tab:metrics}. It is evident that as the number of partitions increases, the replication factor should also increase (with RF equal to 1 for a single partition). Therefore, we present the metrics values for a fixed number of partitions, specifically $651 = 25^{2 \cdot 1} + 25^1 + 1$ for twitter-2010, uk-2002 and com-friendster and $381 = 19^2 + 19 + 1$ for arabic-2005 and graph500-24, graph500-26.

We selected EdgePartition2D from GraphX as the baseline algorithm for comparison.

One can see that replication factor of the partition made by FPP method is approximately 2 times lower, though balance is worse in some cases.

\begin{table}[h]
	\centering
	\begin{tabular}{c c c c c}
		graph & \# parts & partition & balance & RF \\
		& & algorithm & & \\\hline
		uk-2002 & 651 & Edge2D & 1.02 & 21.77 \\
		& & FPP & 1.31 &  10.41\\\hline
		twitter-2010 & 651 & Edge2D & 1.08 & 22.76\\
		& & FPP & 1.26 & 11.59  \\\hline
		com-friendster & 651 & Edge2D & 1.01 & 21.97 \\
		& & FPP & 1.01 & 11.27 \\\hline
		arabic-2005 & 381 & Edge2D & 1.05 & 20.39 \\
		& &  FPP & 1.25 & 10.94 \\\hline
		graph500-24 & 381 & Edge2D & 1.07 & 13.82 \\
		& &  FPP & 1.08 & 7.11 \\\hline
		graph500-26 & 381 & Edge2D & 1.03 & 13.56 \\
		& & FPP & 1.06 & 6.97 \\\hline
	\end{tabular}
	\caption{\label{tab:metrics}A metrics table.}
\end{table}

\subsection{Performance improvement}

Our implementation of FPP Partition was also written on Spark. Some optimizations were made for partition process in both partitioning methods, our and baseline, that made possible reducing partition time itself.

We considered as target algorithm Connected Components algorithm implemented inner on GraphX, and also PageRank.

We identified the optimal settings for each graph partition algorithm and target algorithm. The tables present a comparison of the best results obtained using EdgePartition2D and FPP partitions with the determined version of the algorithm.

It is important to note that the performance of these algorithms is influenced by the cluster configuration, settings, and the number of partitions. We conducted tests to observe how the runtime varies with an increasing number of partitions. The results are visualized in the graphics shown in Figures \ref{fig:graphics1} and \ref{fig:graphics2}.

\begin{figure}[h]
	\centering
	\includegraphics[width=0.45\textwidth]{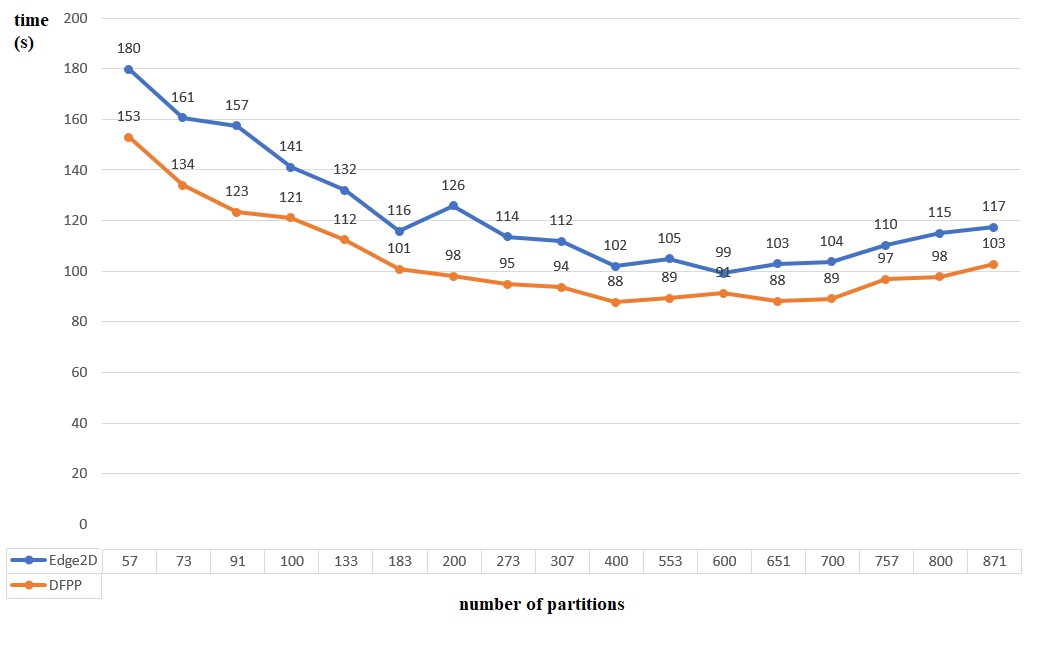}
	\caption{\label{fig:graphics1} Connected components algorithm runtime using Edge2D and DFPP partitions for graph twitter-2010}
\end{figure}

\begin{figure}[h]
	\centering
	\includegraphics[width=0.45\textwidth]{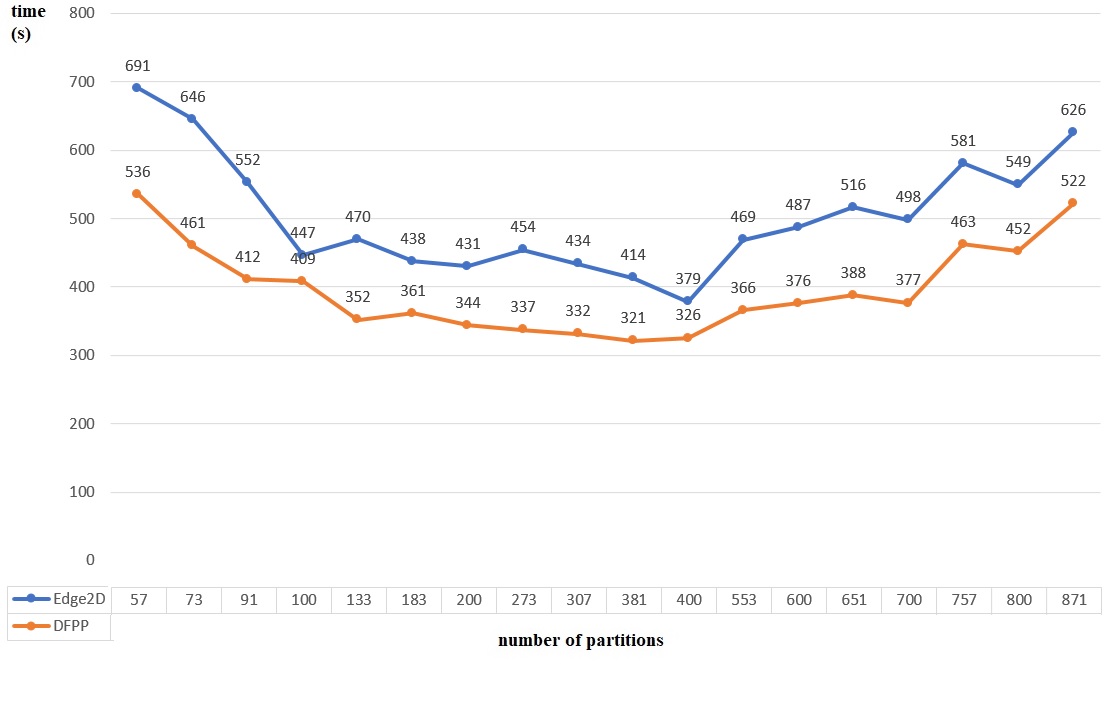}
	\caption{\label{fig:graphics2} PageRank algorithm runtime using Edge2D and DFPP partitions for graph graph500-24}
\end{figure}

It is evident from the results that different partition algorithms have varying optimal numbers of partitions. Therefore, a comparison of partition algorithms was conducted using the respective optimal number of partitions for each algorithm.

So, for example, in the case of accelerating Connected Components on graph datasets twitter-2010 and com-friendster, we used a total of 651 partitions for the FPP partition algorithm, which corresponds to $25^2 + 25 + 1$. For the EdgePartition2D algorithm, we selected 700 partitions for the twitter-2010 graph and 400 partitions for the com-friendster graph, as these settings yielded the best results.

The experiments were conducted on a 4-node \\ Apache Spark cluster comprising 1 master node (driver) and 3 worker nodes. The cluster utilized Kunpeng \\ servers with a total of 96 cores running at 2.6 MHz, 384 GiB of RAM, and a 10Gb/s network.


The runtime in seconds for each algorithm and configuration is provided in Tables \ref{tab:timesPR} and \ref{tab:times}. The improvement ratio in Tables \ref{tab:timesPRImp} and \ref{tab:timeImp} was calculated as the ratio of the runtime for EdgePartition2D (Edge2D) to the runtime for FPP: 
$$
improv. = time_{Edge2D} / time_{FPP}.
$$

The algorithm time improvement is calculated in a similar way: 
$$
alg.time\ improv. = alg.time_{Edge2D} / alg.time_{FPP}.
$$

\begin{table}[h]
	\centering
	\begin{tabular}{c c c c c c c c c}
		
		graph & partition& \# & part. & alg. & end-to \\
		& alg. & parts & time & time & -end t. \\\hline
		
		twitter & Edge2D & 600 & 35.9 & 99.2 & 135.1 \\
		-2010 & FPP & 651 & 34.9 & 87.9 & 122.8 \\\hline
		com- & Edge2D & 500 & 50.6 & 217.3 & 267.8 \\
		friendster & FPP & 651 & 43.5 & 155.5 & 199.0 \\\hline
		uk & Edge2D & 273 & 11.3 & 60.0 & 71.4 \\
		-2002 & FPP & 273 & 11.3 & 51.2 & 62.5 \\\hline
		arabic & Edge2D & 381 & 17.3 & 68.3 & 85.6 \\
		-2005 & FPP & 381 & 17.75 & 62.5 & 80.2 \\\hline
		graph500 & Edge2D & 381 & 27.9 & 39.4  & 67.3 \\
		-26 & FPP & 381 & 27.5 & 35.7  & 63.2 \\\hline
	\end{tabular}
	\caption{\label{tab:timesPR}A runtime (in seconds) table for Connected Components.}
\end{table}

\begin{table}[h]
	\centering
	\begin{tabular}{c c c c c c c c}
		
		graph & alg.time & end-to-end \\
		& improvement & improvement \\\hline
		
		twitter-2010 & 1.13x & 1.10x\\\hline
		com-friendster& 1.40x & 1.35x\\\hline
		uk-2002 & 1.17x & 1.14x \\\hline
		arabic-2005 & 1.09x & 1.07x \\\hline
		graph500-26 & 1.10x & 1.07x \\\hline
	\end{tabular}
	\caption{\label{tab:timesPRImp}Runtime improvement table for Connected Components.}
\end{table}

\begin{table}[h]
	\centering
	\begin{tabular}{c c c c c c c}
		graph & partition& \# & part. & alg. & end-to \\
		& alg. & parts & time & time & -end t. \\\hline
		
		uk-2002 & Edge2D & 200 & 15.5 & 778.8 & 794.3 \\
		& FPP & 307 & 17.0 & 584.2 & 601.2 \\\hline
		graph & Edge2D & 400 & 13.9 & 378.5 & 392.4 \\
		500-24 & FPP & 381 & 40.3 & 321.2 & 335.5 \\\hline
	\end{tabular}
	\caption{\label{tab:times}A runtime (in seconds) table for PageRank.}
\end{table}

\begin{table}[h]
	\centering
	\begin{tabular}{c c c c c c c}
		graph & alg.time  & end-to-end \\
		& improvement & improvement  \\\hline
		
		uk-2002 & 1.33x & 1.32x \\\hline
		graph500-24 & 1.18x & 1.17x \\\hline
	\end{tabular}
	\caption{\label{tab:timeImp}Runtime improvement table for PageRank.}
\end{table}

\section{Conclusion}\label{seq:conclusion}

The proposed algorithm has demonstrated significant performance improvement for the target algorithms on graph data. The partition quality, as measured by the replication factor, has also been improved compared to the baseline algorithm. The theoretical bound of the replication factor achieved by the constrained approach proves its effectiveness.

In the context of a pipeline where the graph is partitioned before each application run, our method offers a relatively fast partitioning process, making it well-suited for such scenarios. However, it can also be applied in situations where partitioning is performed once and the application is executed multiple times.

Overall, the algorithm presents a promising solution for \\ graph partitioning, offering both performance acceleration and improved replication factor, making it applicable in various computational scenarios.







\bibliographystyle{elsarticle-num}



\end{document}